\documentclass{amsart}
\usepackage{amsmath}
\usepackage{amsfonts}
\usepackage{amssymb}
\usepackage{setspace}

\newtheorem{theorem}{Theorem}
\newtheorem{lemma}{Lemma}
\newtheorem{proposition}{Proposition}
\newtheorem{corollary}{Corollary}

\theoremstyle{definition}

\newtheorem*{gromov-lem-1-repeat}{Lemma \ref{gromov-lem-1}}
\newtheorem*{to-prove-prop-repeat}{Proposition \ref{to-prove-prop}}
\newtheorem*{rep-prop-repeat}{Proposition \ref{rep-prop}}

\theoremstyle{remark}

\renewcommand{\leq}{\leqslant}
\renewcommand{\geq}{\geqslant}

\newsavebox{\proofbox}
\savebox{\proofbox}{\begin{picture}(7,7)%
  \put(0,0){\framebox(7,7){}}\end{picture}}

\def\E{\mathbb{E}}
\def\N{\mathbb{N}}
\def\F{\mathbb{F}}
\def\Z{\mathbb{Z}}
\def\R{\mathbb{R}}

\def\C{\mathbb{C}}

\def\eps{\varepsilon}

\makeindex

\begin{document}

\onehalfspace

\title[Bounded depth circuits and M\"obius]{On (not) computing the M\"obius function using bounded depth circuits}

\begin{abstract} Any function $F : \{0,\dots,N-1\} \rightarrow \{-1,1\}$ such that $F(x)$ can be computed from the binary digits of $x$ using a bounded depth circuit is orthogonal to the M\"obius function $\mu$ in the sense that $\frac{1}{N} \sum_{0 \leq x \leq N-1} \mu(x)F(x) \rightarrow 0$ as $N \rightarrow \infty$. The proof combines a result of Linial, Mansour and Nisan with techniques of K\'atai and Harman-Kat\'ai, used in their work on finding primes with specified digits. 
\end{abstract}

\author{Ben Green}

\address{Centre for Mathematical Sciences\\
Wilberforce Road\\
Cambridge CB3 0WA\\
England}
\email{b.j.green@dpmms.cam.ac.uk}

\maketitle

\section{Introduction}

This note arose from a question posed by Gil Kalai on Math Overflow \cite{mo-kalai}, and discussed prior to that on the blogs \cite{kalai-blog,lipton-blog}. This question was motivated by a general programme advanced by Peter Sarnak (see, for example \cite{sarnak}), the aim of which is to rigorously establish instances of the M\"obius randomness principle \cite[p. 338]{iwaniec-kowalski}. 

Recall that the M\"obius function $\mu : \N \rightarrow \{-1,0,1\}$ is defined as follows: $\mu(n) = 0$ if $n$ has a nontrivial square factor, $\mu(1) = 1$, and $\mu(p_1\dots p_k) = (-1)^k$ if $p_1,\dots,p_k$ are distinct primes. For the purposes of this paper set $\mu(0) = 0$. The M\"obius randomness principle asserts that $\mu$ is asymptotically orthogonal to any ``low-complexity'' function $F : \N \rightarrow [-1,1]$ in the sense that 
\[ \frac{1}{N} \sum_{x \leq N} \mu(x) F(x) = o(1).\]

Of course, the notion of ``low-complexity'' is not precisely defined. Sarnak has specifically asked about the case in which $F$ arises from a zero-entropy dynamical system; we note that already the most trivial case $F \equiv 1$ is equivalent to the prime number theorem.

Suppose for simplicity that $N = 2^n$. Our aim here is to establish the M\"obius randomness principle for a class of ``low-complexity'' functions defined in a rather different way, using bounded depth circuits on the $n$ binary digits of $x$. By a boolean circuit in this context we mean a directed acyclic graph in which every node is either an input node of in-degree 0 labelled by one of the binary digits of $x$, an AND gate, an OR or a NOT gate. The AND and OR gates are allowed to have arbitrary fan-in (number of inputs). One of these gates is designated as the output gate. Such a circuit naturally computes a function $F : \{0,\dots, N-1\} \rightarrow \{-1,1\}$.

The \emph{size} of a circuit is the total number of gates it contains and its \emph{depth} is the maximal length of a path from an input gate to the output gate. The class $\mbox{AC}^0(d)$ consists of all functions that can be computed using a circuit of depth at most $d$ and size at most  $n^{d}$.

\begin{theorem}\label{mainthm}
Suppose $N = 2^n$. Let $F : \{0,\dots,N-1\} \rightarrow \{-1,1\}$ be an $\mbox{\emph{AC}}^0(d)$ function. Then 
\[ \E_{0 \leq x \leq N-1} \mu(x) F(x) = O(e^{d \log n - cn^{1/6d}}),\] where $c > 0$ is an absolute constant. 
\end{theorem}
\emph{Remarks.} Here, and throughout the paper, we write $\E_{x \in X} \psi(x)$ for the average of a function $\psi$ over some finite set $X$. Note that the bound is $o(1)$ for $d \leq c'\log n/\log\log n$, and hence in particular for any fixed $d$.\vspace{11pt}

Using a result of Linial, Mansour and Nisan \cite{linial-mansour-nisan}, to be recalled in detail in the next section, the problem can be reduced to that of estimating certain of what Kalai calls the \emph{Fourier-Walsh coefficients} of $\mu$. \vspace{11pt}

\emph{Definition.}
Suppose that $N = 2^n$, and for $x \in \{0,\dots,N-1\}$ write $x = x_1 + 2 x_2 + \dots + 2^{n-1} x_{n}$ in binary. Let $f : \{0,\dots,N-1\} \rightarrow \C$ be a function. Then we define the Fourier-Walsh coefficients $\hat{f}(S)$, $S \subseteq \{1,\dots, n\}$, by 
\[ \hat{f}(S) := \E_{0 \leq x \leq N-1} f(x) (-1)^{\sum_{i \in S} x_i}.\]

We remark that these are nothing more than the Fourier coefficients of $f$, considered as a function on the binary cube $\F_2^n$ instead of as a function on $\{0,\dots,N-1\}$.

The new content of this note is the following bound on Fourier-Walsh coefficients, which (given  the result of Linial, Mansour and Nisan) rather easily implies Theorem \ref{mainthm}.

\begin{proposition}\label{fourier-walsh-bound}
Suppose that $S \subseteq \{1,\dots,n\}$ has size $|S| = k$. Then $\hat{\mu}(S) = O(k e^{-cn^{1/2}/k})$, where $c > 0$ is some absolute constant.\end{proposition}
This bound is nontrivial for $k = O(n^{1/2}/\log n)$, and gives a bound of shape $e^{-\log^{\Omega(1)} N}$ as soon as $k < n^{1/2 - \eta}$. 

It seems to be an interesting question in its own right to ask for a proof that $\hat{\mu}(S) = o(1)$ for all sets $S$. It seems that the methods of this paper cannot hope to handle sets $S$ with $|S|$ bigger than about $n^{1/2}$ unconditionally, although assuming the Generalised Riemann Hypothesis (GRH) we obtain nontrivial bounds for $|S| = O(n/\log n)$. 
Interestingly, the most extreme case $S = \{1,\dots,n\}$ follows from the techniques of Mauduit and Rivat \cite{mauduit-rivat}. It is possible that by combining those ideas with the techniques in this paper one could indeed obtain a bound for all $S$, at least on GRH, but this is by no means a trivial matter. We hope to return to it in a future paper\footnote{Added in proof: Jean Bourgain has since obtained such a bound.}.

In the author's opinion it is very slightly more natural to consider, in place of the M\"obius function $\mu$, the Liouville function $\lambda$. This function is the unique completely multiplicative function such that $\lambda(p) = -1$ for all primes $p$, or equivalently $\lambda(x)$ is $\pm 1$ according to the parity of the number of prime factors of $x$, counted with multiplicity. All of the results in this paper hold equally well for the Liouville function, with very similar proofs.

Finally let it be remarked that there is a considerable literature concerning the difficulty of computing number theoretic functions. Of particular relevance to this paper is the work \cite{ass}, which shows that detecting primality does not lie in $\mbox{AC}^0$, and \cite{bds}, which applies the result of Linial, Mansour and Nisan to show that squarefreeness cannot be detected using bounded depth circuits.

\emph{Acknowledgements.} It is a pleasure to thank Gil Kalai for bringing this question to my attention and for encouragement, and for drawing attention to the work of Linial, Mansour and Nisan. It is also my pleasure to thank Andrew Granville for several helpful conversations concerning the distribution of primes in progressions. Finally, I thank the anonymous referee and the editors of CPC for a careful reading of the paper.

\section{A result of Linial, Mansour and Nisan}

The following result from \cite{linial-mansour-nisan} provides a link between functions in $\mbox{AC}^0$, that is to say computed by bounded depth circuits of polynomial size, and Fourier-Walsh coefficients. 

\begin{theorem}[Linial, Mansour and Nisan]
Suppose that $F : \{0,\dots, N-1\} \rightarrow \{-1,1\}$ is computed using a circuit of depth $d$ and size $M$. Then 
\[ \sum_{|S| > t}|\hat{F}(S) |^2 \leq 2 M 2^{-t^{1/d}/20}.\]
\end{theorem}

Using this, and assuming Proposition \ref{fourier-walsh-bound}, it is not a difficult matter to establish Theorem \ref{mainthm}. We do this now. Let $F : \{0,\dots,N-1\} \rightarrow \{-1,1\}$ be an $\mbox{AC}^0(d)$ function; thus $F$ is computed by a binary circuit of depth at most $d$ and size at most $n^d$. Let $\gamma > 0$ be a small real number.
By Parseval's identity on the cube $\F_2^n$ we have\[ \E_{0 \leq x \leq N-1} \mu(x) F(x)  = \sum_{S} \hat{\mu}(S)\hat{F}(S) .\] Using the trivial bound $|\hat{F}(S)| \leq 1$ for $|S| \leq n^{1/6}$, this is bounded in modulus by
\[ n^{n^{1/6}}\sup_{|S| \leq n^{1/6}}| \hat{\mu}(S)| + \sum_{|S| > n^{1/6}} |\hat{\mu}(S)||\hat{F}(S)|.\]
Applying Proposition \ref{fourier-walsh-bound} to the first term and the Cauchy-Schwarz inequality and Parseval's identity to the second, it follows that this is at most\[  O(e^{-c n^{1/3}})+ \big(\sum_{|S| > n^{1/6}}|\hat{F}(S)|^2\big)^{1/2} .\]
Finally, by the theorem of Linial, Mansour and Nisan, this is no more than
\[  O(e^{-cn^{1/3}})  +  (2 n^d e^{-cn^{1/6d}} )^{1/2} = O(e^{d \log n -  cn^{1/6d}}),\] as required.

\section{An argument of K\'atai}

In this section we give a general version of an argument from \cite{katai} (see also \cite{harman-katai}). In the language of this paper, it gives a link between the Fourier-Walsh coefficients $\hat{f}(S)$ and the ``traditional'' Fourier coefficients
\[ \hat{f}(\theta) := \E_{0 \leq x \leq N-1} f(x) e(\theta x),\] where $\theta \in [0,1]$ and here, as always, $e(u) := e^{2\pi i u}$.  By a \emph{dyadic} rational we mean a rational number with denominator a power of two.

\begin{proposition}\label{lems}
Let $f : \{0,\dots,N-1\} \rightarrow [-1,1]$ be a function for which there exists some $S \subseteq \{1,\dots,n\}$, $|S| = k$, for which the Fourier-Walsh coefficient $\hat{f}(S)$ has magnitude at least $\delta$, $0 < \delta < 1/2$. Then there is a dyadic rational $\theta$ such that the traditional Fourier coefficient $\hat{f}(\theta)$ has magnitude at least $(\delta/10 k)^{4k}$. In fact, $\theta$ can be taken to be a \emph{sparse} dyadic rational in the sense that it may be written as $\theta = \frac{r_1}{2^{i_1}} + \dots + \frac{r_k}{2^{i_k}}$ with $r_i \in \Z$ and $|r_i| \leq (10k/\delta)^3$ for all $i$. 
\end{proposition}
\begin{proof} We have
\begin{equation}\label{eq-11} \hat{f}(S) = \E_{0 \leq x \leq N-1} f(x)\prod_{i \in S} \psi(\frac{x}{2^i}),\end{equation}
where $\psi : \R/\Z \rightarrow [-1,1]$ is the function defined by $\psi(t) = 1$ if $0 \leq t < \frac{1}{2}$ and $\psi(t) = -1$ if $\frac{1}{2} \leq t < 1$.

We shall replace $\psi$ by a smoothed version $\tilde\psi : \R/\Z \rightarrow [-1 ,1]$ with the property that 
\begin{equation}\label{needed} \E_{0 \leq x \leq N-1} |\psi(\frac{x}{2^i}) - \tilde \psi(\frac{x}{2^i}) | \leq \eps\end{equation}for $i = 0,1,2,\dots,n$, and which has a Fourier expansion of the shape

\begin{equation}\label{shape} \tilde\psi(t) = \sum_{|r| \leq 100\eps^{-3}} a_r e(rt)\end{equation} with $|a_r| \leq 1$ for all $r$. 

Such a function may be constructed by a fairly standard type of smoothing procedure, which we now describe. Consider first of all the function 
$\psi_0 := \phi \ast \chi \ast \chi$, where $\phi : \R/\Z \rightarrow [-1,1]$ is defined by $\phi(t) = \psi(t + \frac{\eps}{24})$, that is to say $\phi(t) = 1$ if $-\frac{\eps}{24} \leq t <  \frac{1}{2} - \frac{\eps}{24}$ and $\phi(t) = -1$ if $\frac{1}{2} - \frac{\eps}{24} \leq t <  1 - \frac{\eps}{24}$, and $\chi := \frac{24}{\eps} 1_{[-\eps/48,\eps/48]}$. Note that $\psi_0$ takes values in $[-1,1]$ and that $\psi_0(t) = \psi(t)$ outside of the pairs of intervals $I_1 = [\frac{1}{2} - \frac{\eps}{12}, \frac{1}{2}]$ and $I_2 = [1 - \frac{\eps}{12}, 1]$.

The interval $I_1$, being situated adjacent to but not including the point $\frac{1}{2}$, contains at most its fair share (that is, $\eps/12$) of rationals with denominator $2^i$ for any $i$. The same is true for $I_2$. It follows easily that 
\begin{equation}\label{bzz} \E_{0 \leq x \leq N-1} |\psi(\frac{x}{2^i}) - \psi_0(\frac{x}{2^i}) | < \eps/3.\end{equation}

Now the Fourier coefficients of $\psi_0$ are given by $\hat{\psi}_0(r) = \hat{\phi}(r)\hat{\chi}(r)^2$, and in particular (computing $\hat{\chi}$ explicitly) we have the bound $|\hat{\psi}_0(r)| \leq |\hat{\chi}(r)|^2 \leq \min(1, 24/\eps\pi|r|)^2$. It follows that
\[ \sum_{|r| > 100/\eps^3} |\hat{\psi}_0(r)| < \frac{\eps}{3}.\]
Define $\psi_1(t) := \sum_{|r| \leq 100/\eps^3} \hat{\psi}_0(r) e(rt)$; This clearly has the shape \eqref{shape}. Furthermore from the preceding analysis we have the bound
\begin{equation}\label{tri-1}
 \Vert \psi_1 - \psi_0 \Vert_{\infty}  \leq \sum_{|r| > 100/\eps^3} |\hat{\psi}_0(r)|   < \eps/3.
\end{equation}
This also implies that $\Vert \psi_1\Vert_{\infty} \leq \Vert\psi_0 \Vert_{\infty} + \frac{\eps}{3} \leq 1 + \frac{\eps}{3}$. Finally set $\tilde\psi := \psi_1/(1 + \frac{\eps}{3})$, so that $\tilde\psi$ takes values in $[-1,1]$. Then 
\begin{equation}\label{tri-2} \Vert \tilde\psi - \psi_1 \Vert_{\infty} = \Vert \psi_1 \Vert_{\infty} \big(1 - \frac{1}{1 + \eps/3} \big)\leq \eps/3.\end{equation}

This function $\tilde\psi$ is still of the form \eqref{shape}. Combining \eqref{bzz}, \eqref{tri-1} and \eqref{tri-2} using the triangle inequality gives the desired property \eqref{needed}.

Choose $\eps := \delta/2k$ in this construction. Replacing each copy of $\psi$ by $\tilde\psi$ in turn in \eqref{eq-11} and applying the property \eqref{needed} repeatedly, we obtain

\[ \big|\E_{0 \leq x \leq N-1} f(x) \prod_{i \in S} \tilde\psi (\frac{x}{2^i})\big| \geq \delta - k\eps \geq \delta/2.\] 

Now develop each $\tilde\psi$ in its Fourier series to get

\[ \big| \sum_{|r_1|,\dots,|r_k| \leq 100/\eps^3} a_{r_1} a_{r_2} \dots a_{r_k} \E_{0 \leq x \leq N-1} f(x) e\big( (\frac{r_1}{2^{i_1}} + \dots + \frac{r_k}{2^{i_k}} ) x)\big|  \geq \delta/2,\]
where $S = \{i_1,\dots, i_k\}$.  

It follows immediately that there is some $\theta = \frac{r_1}{2^{i_1}} + \dots + \frac{r_k}{2^{i_k}}$ such that 
\[ |\hat{f}(\theta)| \geq \frac{\delta}{2} (\frac{\eps^3}{300})^k >   (\frac{\delta}{10k})^{4k},\] which is the claimed bound.

The fact that $\theta$ is a \emph{sparse} dyadic rational, in the sense claimed, is clear from the proof.
\end{proof}

On the assumption of the GRH, Proposition \ref{lems} already implies Theorem \ref{mainthm}. Indeed on that hypothesis it is known, by work of Baker and Harman \cite{baker-harman}, that $|\hat{\mu}(\theta)| \ll_{\eps} N^{-1/4 + \eps}$ for all $\eps > 0$. We note that this together with Proposition \ref{lems} gives a nontrivial estimate $\hat{\mu}(S) = o(1)$ for all $S \subseteq \{1,\dots,n\}$ with $|S| \leq cn/\log n$. 

Without some unproved hypothesis, however, the best available bounds for general $\theta$ are $|\hat{\mu}(\theta)| \ll_A \log^{-A} N$ for all $A > 0$, the implied constant being ineffective as a function of $A$. This 1937 result of Davenport \cite{davenport-1937}, whilst it leads to a nontrivial bound on $\hat{\mu}(S)$ for $|S| \leq A\log n/\log \log n$ for all $A$, does not suffice for our main application.

To proceed further, some special use must be made of the sparse dyadic structure of $\theta$. In this regard we will prove, in the next section, the following result. 

\begin{proposition}\label{to-prove-prop}
There is some $c_1 > 0$ with the following property. Let $N = 2^n$, suppose that $k < n^{1/2}$, and suppose that $\theta = \frac{r_1}{2^{i_1}} + \dots + \frac{r_k}{2^{i_k}}$ with  $|r_i| \leq e^{c_1\sqrt{\log N}}$ for all $i$. Then $|\hat{\mu}(\theta)| = O( e^{-c_1\sqrt{\log N}})$.
\end{proposition}

\emph{Proof of Proposition \ref{fourier-walsh-bound}.}  Suppose that Proposition \ref{fourier-walsh-bound} fails. Then there is some $S \subseteq \{1,\dots,n\}$ such that $|\hat{\mu}(S)| \geq k e^{-cn^{1/2}/k}$. By Proposition \ref{lems} there is a sparse dyadic rational $\theta = \frac{r_1}{2^{i_1}} + \dots + \frac{r_k}{2^{i_k}}$, $|r_i| \ll e^{3c n^{1/2}/k}$, such that $|\hat{\mu}(\theta)| \geq e^{-4cn^{1/2}}$. If $c$ is chosen small enough, this is contrary to Proposition \ref{to-prove-prop}.

\section{Exponential sums over M\"obius at sparse dyadic rationals}

In this section we complete our remaining task, which is to establish Proposition \ref{to-prove-prop}. We begin by quoting the following result, which would be well-known to experts in analytic number theory, but for which it is hard to find a really concise reference.

\begin{theorem}\label{quote-thm}
For some absolute constant $c_2 > 0$ the following is true. Suppose that $\chi$ is a Dirichlet character to modulus $q = 2^t$, $q \leq e^{c_2\sqrt{\log N}}$. Then 
\[ \E_{0 \leq x \leq N-1} \mu(x) \chi(x) = O(e^{-c_2\sqrt{\log N}}).\]
\end{theorem}
\begin{proof} 
By standard techniques of analytic number theory (Perron's formula and the classical zero-free region for Dirichlet $L$-functions) one may establish the bound
\[ \E_{0 \leq x \leq N-1} \mu(x) \chi(x) = O(e^{-c_2\sqrt{\log N}})\]
for \emph{all} Dirichlet characters $\chi$ to modulus $q \leq e^{c_2\sqrt{\log N}}$, provided only that $L(s,\chi)$ has no \emph{exceptional zero}. This last possibility only occurs (if at all) when $\chi$ is real, and the zero in question lies on the segment $(1/2,1]$. 

For a proof of this bound the reader may consult the very comprehensive book \cite{montgomery-vaughan}, in which this precise statement is Exercise 11.3.7; the key ingredient, other than Perron's formula, is the upper bound for $1/L(s,\chi)$ given by Theorem 11.4 of that book, and specifically equation (11.7).

We now specialise to the case $q = 2^t$. Without loss of generality we may suppose that $q$ is the conductor of $\chi$. To rule out the possibility of exceptional zeros note that there \emph{are} only three nonprincipal primitive real Dirichlet characters with conductor a power of two: the character $\chi_4$ on $(\Z/4\Z)^*$ with $\chi_4(-1) = -1$, and the characters $\chi_8$ and $\chi_4 \chi_8$, where $\chi_8$ is the primitive character on $(\Z/8\Z)^*$ defined by $\chi_8(3) = \chi_8(5) = -1$ and $\chi_8(1) = \chi_8(7) = 1$. This follows from the well-known fact that 
\[ (\Z/2^e \Z)^* \cong \{\pm 1\} \times \{1, 5,5^2,\dots, 5^{2^{e-2} - 1}\}\] for $e \geq 3$, which means that any Dirichlet character of conductor dividing $2^e$ is determined by its values at $-1$ and $5$. 

For each of these three characters it is known that there are no real zeros of the corresponding $L(s,\chi)$ on $(\frac{1}{2},1]$; see for example \cite{ramare-rumely}. (In fact it suffices to know that $L(1,\chi) \neq 0$, which is of course true for \emph{all} Dirichlet characters $\chi$. This implies that $L(s,\chi) \neq 0$ for $s \in [1-c', 1]$ for some $c'$ by continuity, which already means that there are no exceptional zeros.)
\end{proof}

\begin{corollary}\label{dyadic-cor}
Let $c_3 := c_2/2$. Then we have the estimate
\[ \E_{0 \leq x \leq N-1} \mu(x) e(a x/2^t) = O(e^{-c_3\sqrt{\log N}}),\] uniformly for $2^t \leq e^{c_3\sqrt{\log N}}$ and $a \in \Z$. 
\end{corollary}
\begin{proof} We use Fourier analysis on $(\Z/2^t\Z)^*$. If $x \in (\Z/2^t\Z)^*$  then $\sum_{\chi} \chi(x)$ equals $2^{t-1}$ if $x = 1$ and vanishes otherwise, where the sum is over all characters $\chi : (\Z/2^t\Z)^*\rightarrow \C^*$. It follows that 
\[ e(\frac{ax}{2^t}) = \frac{1}{2^{t-1}} \sum_{r \in (\Z/2^t \Z)^*} \sum_{\chi} \overline{\chi(r)} e(\frac{ar}{2^t}) \chi(x)\] for all odd $x$. 

It follows from the preceding equation, the triangle inequality and Theorem \ref{quote-thm} that 
\[ \sum_{\substack{0 \leq x \leq N-1\\ 2 \nmid x}} \mu(x) e(\frac{ax}{2^t}) \leq 2^{t-1} \sup_{\chi}|\sum_{0 \leq x \leq N-1} \mu(x) \chi(x)| = O( N e^{-c_3\sqrt{\log N}}).\]
We could have estimated the Gauss sum here less trivially, but there is no need.
By the same argument (using the fact that $\mu(2x) = -\mu(x)$ when $x$ is odd) we have
\[ \sum_{\substack{0 \leq x \leq N-1\\ 2 | x\\ 4 \nmid x}} \mu(x) e(\frac{ax}{2^t}) = -\sum_{0 \leq x < \frac{1}{2}N: 2 \nmid x} \mu(x) e(\frac{ax}{2^{t-1}}) = O(N e^{-c_3\sqrt{\log N}}).\]
This concludes the proof of the corollary.
\end{proof}

Finally we give a further standard deduction, allowing us to bound $\hat{\mu}(\theta)$ when $\theta$ is \emph{near} a dyadic rational.

\begin{corollary}\label{near-dyadic}
Let $c_4 := c_3/3$. Suppose that $|\theta - a/q| \leq e^{c_4 \sqrt{\log N}}/N$, where $q \leq e^{c_4\sqrt{\log N}}$ is a power of two. Then $\hat{\mu}(\theta) = O(e^{-c_4\sqrt{\log N}})$.
\end{corollary}
\begin{proof}
Let $L$, $1 < L < N$, be a quantity to be specified. If $P \subseteq \{0,\dots,N-1\}$ is any interval of integers of length $L$ and if $x_0$ lies in $P$ then we have
\begin{align*}| \sum_{x \in P} \mu(x) e(\theta x)| & = |\sum_{x \in P} \mu(x) e(ax/q) e((x - x_0)(\theta - a/q))| \\ & = |\sum_{x \in P} \mu(x) e(ax/q)| + O(\frac{L^2}{N}e^{c_4\sqrt{\log N}})  \\ & = O(N e^{-c_3\sqrt{\log N}}) + O(\frac{L^2}{N}e^{c_4\sqrt{\log N}}) ,\end{align*} the latter estimate following immediately from Corollary \ref{dyadic-cor}. Dividing $\{0,\dots,N-1\}$ into $O(N/L)$ such progressions we obtain
\[ \hat{\mu}(\theta) = O(\frac{N}{L}e^{-c_3\sqrt{\log N}}) + O(\frac{L}{N} e^{c_4\sqrt{\log N}}).\]
Choosing $L := Ne^{-2c_4\sqrt{\log N}}$ gives the result.
\end{proof}

Next, we recall a standard estimate on exponential sums over the M\"obius function. This may be established using the method of Type I/II sums or bilinear forms, perhaps most easily by utilising Vaughan's identity. The estimate provides an excellent bound for $\hat{\mu}(\theta)$ if $\theta$ is not well-approximated by a rational with small denominator.

\begin{proposition}\label{vaughan}
Suppose that $|\hat{\mu}(\theta)| \geq \delta$. Then there is some $q \ll ( \frac{\log N}{\delta})^{16}$ such that $|\theta - a/q| \ll (\frac{\log N}{\delta})^{16}N^{-1}$.
\end{proposition}
\begin{proof} (Sketch.) Let $Q = cN(\frac{\log N}{\delta})^{-16}$ for some small $c > 0$ to be specified shortly. By Dirichlet's theorem on diophantine approximation there is some $q$, $1 \leq q \leq Q$, such that $|\theta - a/q| \leq 1/qQ$. In particular $|\theta - a/q| \leq 1/q^2$, and so we may apply \cite[Theorem 13.9]{iwaniec-kowalski} to conclude that $\delta(\log N)^{-4} \ll \max(q^{1/4} N^{-1/4}, q^{-1/4}, N^{-1/10})$.  If $\delta < N^{-1/12}$ (say) then the result is trivial, so we have either $\delta (\log N)^{-4} \ll q^{1/4} N^{-1/4} \leq Q^{1/4} N^{-1/4}$ or $\delta (\log N)^{-4} \ll q^{-1/4}$. If $c$ is small enough then the first of these is impossible by the choice of $Q$, and so the second must hold. This implies that $q \ll (\frac{\log N}{\delta})^{16}$, thereby concluding the proof.
\end{proof}

Corollary \ref{near-dyadic} and Proposition \ref{vaughan} between them provide good estimates on $\hat{\mu}(\theta)$ in two cases: when $\theta$ is close to a rational with denominator a power of two, and when $\theta$ is not very close to any rational with small denominator.

Of course, there are values of $\theta$ (such as $\theta = \frac{1}{3}$) not covered by either proposition. However we have the following key observation of Harman and K\'atai \cite{harman-katai}, allowing one to conclude that none of these exceptional values of $\theta$ are sparse dyadic rationals. Here is a precise formulation of their idea.

\begin{lemma}\label{dio-lemma}
Suppose that $\theta = \frac{r_1}{2^{i_1}} + \dots + \frac{r_k}{2^{i_k}}$, where $i_1 < \dots < i_k \leq n$ and $|r_i| \leq Q$ for all $i$. Suppose furthermore that there is some $q \leq Q$ and an $a$ coprime to $q$ such that $|\theta - a/q| \leq Q/N$, and that we have $2^{n/2k} > 4Q^2$. 
Then $q$ is a power of two.
\end{lemma}
\begin{proof}
Set $i_0 := 0$ and $i_{k+1} := n$. By the pigeonhole principle there is some index $j \in \{0,1,\dots,k\}$ such that $|i_j - i_{j+1}| \geq n/2k$. Now we note that $\theta$ is very close to a rational with denominator $q' = 2^{i_{j}}$; specifically, writing $a' := r_{1}2^{i_j - i_1} + \dots + r_{j-1}2^{i_j - i_{j-1}} + r_{j}$, we have
\[ |\theta - \frac{a'}{q'}| \leq Q(2^{- i_{j+1}} + 2^{- i_{j+2}} + \dots ) \leq 2^{-n/2k} \frac{2Q}{q'}.\] 
Note also that $q' \leq 2^{-n/2k} N$, and so (by the hypothesis) $Q/N < 1/4Qq'$.  It follows that
\[ |\frac{a}{q} - \frac{a'}{q'}| \leq |\theta - \frac{a'}{q'}| + |\theta - \frac{a}{q}| \leq 2^{-n/2k} \frac{2Q}{q'} + \frac{Q}{N} < \frac{1}{2Qq'} + \frac{Q}{N}  < \frac{1}{Qq'} \leq \frac{1}{qq'},\]
which of course implies that $q = q'$ and $a = a'$.
\end{proof}

We are now ready to prove Proposition \ref{to-prove-prop} which, as we remarked earlier, completes the proof of our main theorem.  Let us recall the statement.

\begin{to-prove-prop-repeat}
There is some $c_1 > 0$ with the following property. Let $N = 2^n$, suppose that $k < n^{1/2}$, and suppose that $\theta = \frac{r_1}{2^{i_1}} + \dots + \frac{r_k}{2^{i_k}}$ with  $|r_i| \leq e^{c_1\sqrt{\log N}}$ for all $i$. Then $\hat{\mu}(\theta) = O( e^{-c_1\sqrt{\log N}})$.
\end{to-prove-prop-repeat}
\begin{proof}
Set $\delta := e^{-c_1\sqrt{\log N}}$ and take $c_1 = c_4/32$. Suppose that there is some $\theta$ of the above sparse dyadic form with $|\hat{\mu}(\theta)| \geq \delta$. By Proposition \ref{vaughan} (and the fact that $\delta$ is much smaller than $1/\log N$) there is some $q \ll \delta^{-32} = e^{c_4\sqrt{\log N}}$ such that $|\theta - a/q| \ll e^{c_4\sqrt{\log N}}N^{-1}$. Consider Lemma \ref{dio-lemma} with $Q = e^{c_4\sqrt{\log N}}$. Clearly $|r_i| \leq Q$ for all $i$, and the hypothesis $2^{n/2k} > 4Q^2$ is satisfied. It follows that $q$ is a power of two.

But then Proposition \ref{near-dyadic} may be applied to conclude that $\hat{\mu}(\theta) = O(e^{-c_3\sqrt{\log N}})$, which is what we wanted to prove.
\end{proof}


\begin{thebibliography}{99}


\bibitem{ass} E.~Allender, M.~Saks and I.~Shparlinski, \emph{A lower bound for primality,} Special issue on the Fourteenth Annual IEEE Conference on Computational Complexity (Atlanta, GA, 1999), J. Comput. System Sci. \textbf{62} (2001), no. 2, 356--366. 



\bibitem{baker-harman} R.~C.~Baker and G.~Harman, \emph{Exponential sums formed with the M\"obius function,} J. London Math. Soc. (2) \textbf{43} (1991), no. 2, 193--198.

\bibitem{bds} A.~Bernasconi, C.~Damm and I.~Shparlinski, \emph{On the average sensitivity of testing square-free numbers,} Computing and combinatorics (Tokyo, 1999), 291--299, Lecture Notes in Comput. Sci \textbf{1627} Springer, Berlin, 1999. 


\bibitem{davenport-1937} H.~Davenport, \emph{On some infinite series involving arithmetical functions. II,} Quart. J. Math. Oxford \textbf{8} (1937), 313--320.

\bibitem{harman-katai} G.~Harman and I.~K\'atai, \emph{Primes with preassigned digits. II,} Acta Arith. \textbf{133} (2008), no. 2, 171Ð184.

\bibitem{iwaniec-kowalski} H.~Iwaniec and E.~Kowalski, \emph{Analytic number theory,} American Mathematical Society Colloquium Publications \textbf{53}, American Mathematical Society, Providence, RI, 2004. xii+615 pp.

\bibitem{kalai-blog} G.~Kalai, \emph{The AC0 Prime Number Conjecture,} blog post (2011), available at \\
http://gilkalai.wordpress.com/2011/02/21/the-ac0-prime-number-conjecture/.


\bibitem{mo-kalai} G.~ Kalai, \emph{Walsh Fourier transform of the M\"obius function,} Math Overflow question (2011), available at http://mathoverflow.net/questions/57543/walsh-fourier-transform-of-mobius-functions

\bibitem{katai}  I.~K\'atai, \emph{Distribution of digits of primes in $q$-ary canonical form,} Acta Math. Hungar. \textbf{47} (1986), no. 3-4, 341--359. 


\bibitem{linial-mansour-nisan} N.~Linial, Y.~Mansour and N.~Nisan, \emph{Constant depth circuits, Fourier transform, and learnability,} J. Assoc. Comput. Mach. \textbf{40} (1993), no. 3, 607--620. 

\bibitem{lipton-blog} R.~J.~Lipton, \emph{The depth of the M\"obius function,} blog post (2011), available at \\ http://rjlipton.wordpress.com/2011/02/23/the-depth-of-the-mobius-function/

\bibitem{mauduit-rivat} C.~Mauduit and J.~Rivat, \emph{Sur un probl\'eme de Gelfond: la somme des chiffres des nombres premiers,} Ann. of Math. (2) \textbf{171} (2010), no. 3, 1591--1646.

\bibitem{montgomery-vaughan} H.~Montgomery and R.~C.~Vaughan, \emph{Multiplicative number theory. I. Classical theory.,}
Cambridge Studies in Advanced Mathematics \textbf{97}, Cambridge University Press, Cambridge, 2007. xviii+552 pp


\bibitem{ramare-rumely} O.~Ramar\'e and R.~Rumely, \emph{Primes in Arithmetic Progressions,} Mathematics of Computation \textbf{65} (1996), no. 213, 397--425.

\bibitem{sarnak} P.~Sarnak, \emph{M\"obius randomness and dynamics,} lecture notes (2010), available at \\
http://www.math.princeton.edu/sarnak/Mobius lectures Summer 2010.pdf.




\end{thebibliography}
\end{document}